\documentclass[12pt]{amsart}
\usepackage{amssymb}
\usepackage{amsthm}
\usepackage{amsmath}
\usepackage[dvips]{epsfig}
\usepackage{graphicx}
\usepackage{hyperref}
\usepackage{color}
\usepackage{url}
\usepackage{tikz}
\usepackage{caption}
\usepackage{subcaption}
\usepackage{epstopdf}
\usepackage{mathtools}
\usepackage[arrow, matrix, curve]{xy}
\usepackage{marginnote}
\usetikzlibrary{decorations.markings}

\makeatletter\@addtoreset {equation}{section}\makeatother

\newtheorem{theorem}{Theorem}

\newtheorem{lemma}{Lemma}[section]
\theoremstyle{remark}

\theoremstyle{definition}

\theoremstyle{corollary}

{\begin{trivlist} \item[]{\em Proof }}%
{\hspace*{\fill}$\rule{.3\baselineskip}{.35\baselineskip}$\end{trivlist}}

\begin{document}

\title[Peaked periodic waves in the Camassa--Holm equation]{\bf Growth of perturbations
to the peaked periodic waves \\ in the Camassa--Holm equation}

\author{Aigerim Madiyeva}
\address[A. Madiyeva]{Department of Mathematics and Statistics, McMaster University,
Hamilton, Ontario, Canada, L8S 4K1}
\email{madiyeva@mcmaster.ca}

\author{Dmitry E. Pelinovsky}
\address[D. Pelinovsky]{Department of Mathematics and Statistics, McMaster University,
Hamilton, Ontario, Canada, L8S 4K1}


\begin{abstract}
Peaked periodic waves in the Camassa--Holm equation are revisited. Linearized evolution
equations are derived for perturbations to the peaked periodic waves
and linearized instability
is proven both in $H^1$ and $W^{1,\infty}$ norms. Dynamics of perturbations in $H^1$
is related to the existence of two conserved quantities and is bounded in the full nonlinear
system due to these conserved quantities. On the other hand, 
perturbations
to the peaked periodic wave grow in $W^{1,\infty}$ norm
and may blow up in a finite time in the nonlinear evolution of the Camassa--Holm equation.
\end{abstract}

\date{\today}
\maketitle


\section{Introduction}

We address the Camassa-Holm (CH) equation given by
\begin{equation}
\label{CH}
u_t - u_{txx} + 3 u u_x = 2 u_x u_{xx} + u u_{xxx},
\end{equation}
in the setting of $\mathbb{T} := [-\pi,\pi]$ subject to the periodic boundary conditions at $\pm \pi$. Following the original derivation in \cite{CH,CHH}, 
the Camassa--Holm equation arises in hydrodynamical applications as a model for propagation of unidirectional shallow water waves \cite{CL,Johnson}. A generalized version of this equation also models propagation of nonlinear waves inside a cylindrical hyper-elastic rod with a small diameter \cite{Dai}. The CH equation can be interpreted geometrically in terms of geodesic flows on the diffeomorphism group \cite{Kouranbaeva,Mis1}. Generalizations of the CH equation with multi-peakons were constructed in \cite{Anco,Anco2}.

Let $\varphi \in H^1_{\rm per}(\mathbb{T})$ be the Green function satisfying
\begin{equation}
(1 - \partial_x^2) \varphi = 2 \delta_0, \quad x \in \mathbb{T},
\label{Green}
\end{equation}
with $\delta_0$ being Dirac delta distribution centered at $x = 0$.
The CH equation (\ref{CH}) can be rewritten in the convolution form
\begin{equation}
\label{CHconv}
u_t + u u_x + \frac{1}{2} \varphi' \ast \left( u^2 + \frac{1}{2} u_x^2 \right) = 0,
\end{equation}
where $(f \ast g)(x) := \int_{\mathbb{T}} f(x-y) g(y) dy$ denotes the convolution operator
and $\varphi'$ denotes piecewise continuous derivative of $\varphi$ in $x$.
The Green function $\varphi$ can be expressed explicitly in the form
\begin{equation}
\label{Green-explicit}
\varphi(x) = \frac{\cosh(\pi - |x|)}{\sinh(\pi)}, \quad x \in \mathbb{T},
\end{equation}
which shows that $\varphi \in H^1_{\rm per}(\mathbb{T}) \cap W^{1,\infty}(\mathbb{T})$
is a piecewise $C^1$ function with the maximum at
\begin{equation}
\label{maximum}
M := \varphi(0) = \coth(\pi)
\end{equation}
and the minima at
\begin{equation}
\label{minimum}
m := \varphi(\pm \pi) = {\rm csch}(\pi).
\end{equation}
The central peak of $\varphi$
is located at $x = 0$ with $\varphi'(0^{\pm}) = \mp 1$, from which $\varphi$ is
monotonically decreasing towards the turning points at $x = \pm \pi$ where $\varphi'(\pm \pi) = 0$. The graph of $\varphi$ versus $x$ is shown on Figure \ref{peaked-wave}.

\begin{figure}[htb!]
	\includegraphics[width=0.6\textwidth]{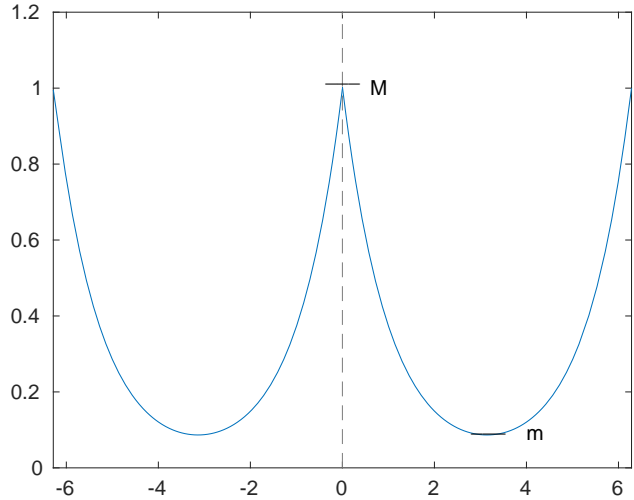}
	\caption{The graph of $\varphi$ on $[-2\pi,2\pi]$ with $M$ and $m$ given by (\ref{maximum}) and (\ref{minimum}).} \label{peaked-wave}
\end{figure} 

The Green function $\varphi$ determines also the travelling periodic wave solution $u(t,x) = \varphi(x-ct)$
to the CH equation (\ref{CHconv}) rewritten in the weak form, where $c$ is the wave speed. The wave speed is uniquely determined by $c = M$ which gives the unique solution 
\begin{equation}
\label{unique-wave}
u(t,x) = \varphi(x-Mt). 
\end{equation}
By using the elementary scaling transformation, this unique solution generates a family of travelling periodic waves of the form:
\begin{equation}
\label{peaked-traveling-wave}
u(t,x) = \gamma \varphi(x-\gamma M t), \quad \gamma \in \mathbb{R}.
\end{equation}
Since the travelling periodic waves (\ref{peaked-traveling-wave}) have a peaked profile $\varphi$,
we call them {\em the peaked periodic waves}. Stability of such peaked periodic waves in $H^1_{\rm per}(\mathbb{T})$ was studied by J. Lenells in \cite{Len1,Len2}
after similar results on stability of peaked solitary waves in $H^1(\mathbb{R})$ in \cite{CM,CS}.

The initial-value problem for the periodic CH equation (\ref{CHconv}) is locally well-posed in the space $H^3_{\rm per}(\mathbb{T})$ \cite{CE,CE2000}, 
$H^s_{\rm per}(\mathbb{T})$ with $s > \frac{3}{2}$ \cite{D,HM2002}, $C^1_{\rm per}(\mathbb{T})$ \cite{Mis2}, and $H^1_{\rm per}(\mathbb{T}) \cap {\rm Lip}(\mathbb{T})$ \cite{DKT}, where ${\rm Lip}(\mathbb{T})$ stands 
for Lipshitz continuous functions. Local well-posedness includes the existence, uniqueness, and continuous dependence of solutions from the initial data. 
Peaked periodic waves and their perturbations can be considered in the space of lower regularity $H^1_{\rm per}(\mathbb{T}) \cap {\rm Lip}(\mathbb{T})$.

{\em Cusped periodic waves} (given by bounded functions with unbounded derivatives on each side of the peaks) exist in $H^1_{\rm per}(\mathbb{T})$. These cusped waves were used in \cite{B,Him} to show that local solutions to the initial-value problem for the periodic CH equation (\ref{CHconv}) are not uniformly continuous 
with respect to the initial data. The same norm inflation occurs
in $H^s_{\rm per}(\mathbb{T})$ for every $s \leq \frac{3}{2}$ \cite{Molinet}. 
The initial-value problem for cusped periodic waves and their perturbations in $H^1_{\rm per}(\mathbb{T})$  is ill-posed due to the lack of continuous dependence on the initial data.

Two energy quantities are well-defined for solutions to the CH equation (\ref{CHconv}) in $H^1_{\rm per}(\mathbb{T})$:
\begin{equation}
\label{conserved-quantities}
E(u) = \int_{\mathbb{T}} (u^2 + u_x^2) dx, \quad F(u) = \int_{\mathbb{T}} u (u^2 + u_x^2) dx.
\end{equation}
These quantities are constant continuously in time before the first time instance, for which the $W^{1,\infty}$-norm of the solution blows up.

By considering perturbations to the peaked periodic wave in $H^1_{\rm per}(\mathbb{T})$ and
using conservation of the mean value as well as conservation of the two energy quantities (\ref{conserved-quantities}),
it was proven in \cite{Len1,Len2} that the peaked periodic wave $u(t,x) = \varphi(x-ct)$ is orbitally stable
in the following sense.

\begin{theorem}\cite{Len1,Len2}
\label{theorem-CS}
For every $\varepsilon > 0$, there is a $\nu > 0$ such that if $u \in C([0,T),H^1_{\rm per}(\mathbb{T}))$
is a solution to the CH equation (\ref{CH}) with the initial data $u_0$ satisfying
\begin{equation}
\label{H1-initial}
\| u_0 - \varphi \|_{H^1(\mathbb{T})} < \nu,
\end{equation}
then
\begin{equation}
\label{H1-final}
\| u(t,\cdot) - \varphi(\cdot - \xi(t)) \|_{H^1(\mathbb{T})} < \varepsilon, \quad t \in [0,T),
\end{equation}
where $\xi(t) \in \mathbb{T}$ is a point where the function $u(t,\cdot)$ attains its maximum on $\mathbb{T}$
and $T > 0$ is either finite or infinite.
\end{theorem}

The main purpose of this work is to clarify the meaning of the stability result of Theorem \ref{theorem-CS}. We will 
prove that the peaked periodic waves are strongly unstable with respect to perturbations in $H^1_{\rm per}(\mathbb{T}) \cap  W^{1,\infty}(\mathbb{T})$. 
In this regard, Theorem \ref{theorem-CS} just ensures that the $H^1$-norm of the solution does not
drift far from the $H^1$-norm of a translated peaked periodic wave,
when the $W^{1,\infty}$ norm grows and even blows up in a finite time.

The following theorem represents the main result of this work.

\begin{theorem}
\label{theorem-nonlinear}
For every $\delta > 0$, there exist $t_0 > 0$ and $u_0 \in H^1_{\rm per}(\mathbb{T}) \cap W^{1,\infty}(\mathbb{T})$
satisfying
\begin{equation}
\label{initial-bound-theorem}
\| u_0 - \varphi \|_{H^1(\mathbb{T})} + \| u_0' - \varphi' \|_{L^{\infty}(\mathbb{T})} < \delta,
\end{equation}
such that the local solution $u \in C([0,T),H^1_{\rm per}(\mathbb{T})\cap W^{1,\infty}(\mathbb{T}))$
to the CH equation (\ref{CHconv}) with the initial data $u_0$ and $T > t_0$ satisfies
\begin{equation}
\label{final-bound-theorem}
\| u_x(t_0,\cdot) - \varphi'(\cdot - \xi(t_0)) \|_{L^{\infty}(\mathbb{T})} \geq 1,
\end{equation}
where $\xi(t) \in \mathbb{T}$ is a point of peak of the function $u(t,\cdot)$ on $\mathbb{T}$.
Moreover, there exist $u_0$ such that the maximal existence time $T$ is finite.
\end{theorem}

The proof of Theorem \ref{theorem-nonlinear} is based on several recent developments.
A similar theorem for peaked solitary waves of the CH equation on an infinite line
was proven in \cite{NP}, to which the present work has many common points.
Analogous study was performed for peaked solitary waves in a different model (the Novikov equation)
in \cite{CP}, where the local well-posedness result of the initial-value problem
in $H^1(\mathbb{R}) \cap W^{1,\infty}(\mathbb{R})$ was not previously available.

In the context of the peaked periodic waves, linear instability of peaked periodic waves was obtained for
a different model (the reduced Ostrovsky equation) in \cite{GP2}.
More recently, spectral instability for perturbations in $L^2(\mathbb{T})$
was obtained in \cite{GP3} for the reduced Ostrovsky equation with either quadratic or cubic nonlinearities.
Neither local well-posedness nor the nonlinear instability was considered in the
framework of the reduced Ostrovsky equation in \cite{GP2,GP3}.

Let us explain the organization of the paper.

In Section 2, we define weak solutions to the CH equation (\ref{CHconv})
in $H^1_{\rm per}(\mathbb{T}) \cap W^{1,\infty}(\mathbb{T})$.
For solutions with a single peak at $\xi(t)$ on $\mathbb{T}$, we prove that the
single peak propagates with the local characteristic speed so that
\begin{equation}
\label{characteristic-speed}
\frac{d \xi}{dt} = u(t,\xi(t)).
\end{equation}
This allows us to define the precise form of the evolution equations for perturbations to the peaked periodic wave (see Appendix A). It follows from (\ref{characteristic-speed}) that the speed $c$ of the travelling peaked periodic wave $u(t,x) = \varphi(x-ct)$
is uniquely determined as $c = \varphi(0) = M$. The latter fact was missed in
the previous works \cite{Len1,Len2}, where $c = 1$ was suggested.

In Section 3, we study the linearized evolution equations for perturbations to the peaked periodic wave.
Similarly to \cite{NP} and \cite{CP}, we are able to simplify the linearized evolution equation
in a compact form, which requires no convolution integrals and which can be solved explicitly
by using the method of characteristics (see Appendix B). Linear instability of the peaked periodic wave
is proven both in $H^1$ and $W^{1,\infty}$. Compared to the previous work \cite{NP}, we add a new result, where we show 
that the linearized instability in $H^1$ is related to the conservation of the two energy quantities
(\ref{conserved-quantities}).

In Section 4, we consider nonlinear dynamics of peaked perturbations to the peaked periodic wave. 
Although the $H^1$ norm of the perturbation does not grow in the nonlinear evolution 
due to to the same two conserved quantities (\ref{conserved-quantities}) as is shown in \cite{Len1,Len2}, 
we prove that the $W^{1,\infty}$ norm of the perturbation can grow. Moreover, we prove that the 
$W^{1,\infty}$ norm of the perturbation can blow up in a finite time as follows:
\begin{equation}
\label{blow-up-criterion}
u_x(t,x) \to -\infty \quad \mbox{\rm at some} \;\; x \in \mathbb{T} \quad \mbox{\rm as} \;\; t \to T^-.
\end{equation}
The wave breakdown criterion (\ref{blow-up-criterion}) is natural for the inviscid Burgers equation
\begin{equation}
\label{Burgers}
u_t + u u_x = 0,
\end{equation}
which contributes to the local part of the CH equation (\ref{CHconv}). 
The precise blow-up rate was derived for the strong solutions in $H^3_{\rm per}(\mathbb{T})$ by using the method of characteristics \cite{CE2000}.
Compared to
the previous work \cite{NP}, we do not take for granted the existence of local solutions in
$H^1_{\rm per}(\mathbb{T}) \cap W^{1,\infty}(\mathbb{T})$ 
and prove local well-posedness for piecewise $C^1$ perturbations with a single peak on $\mathbb{T}$ by using the method of characteristics.
We also show that 
the wave breakdown (\ref{blow-up-criterion}) does occur as a result of the nonlinear instability
of Theorem \ref{theorem-nonlinear}.

Similar to the case of peaked solitary waves in the Camassa-Holm equation \cite{NP} 
and contrary to the case of peaked solitary waves in the Novikov equation \cite{CP}, 
we confirm that the passage from the linear to the nonlinear theory 
is false in $H^1$ as was expected in \cite{CM}. 
Although the linearized instability in $H^1$ is replaced by 
the nonlinear stability result of Theorem \ref{theorem-CS} in $H^1$.
the linearized instability in $W^{1,\infty}$ persists
as the nonlinear instability result of Theorem \ref{theorem-nonlinear}. 
This nonlinear instability arises due to the negative slopes at the right 
side of the peak, which grow unboundedly from below. 

Finally, we mention two open problems for the CH equation (\ref{CH}). 

Peaked periodic waves (\ref{peaked-traveling-wave}) appear to be at the border between two families of smooth and cusped periodic waves (see \cite{Len3} and Appendix C). Existence of smooth periodic waves was studied in \cite{GV}, whereas cusped periodic waves were analyzed in \cite{B,Him}. It is naturally to expect that the smooth periodic waves are orbitally stable in $H^1_{\rm per}(\mathbb{T}) \cap W^{1,\infty}(\mathbb{T})$ similarly to the case of the reduced 
Ostrovsky equation considered in \cite{GP1}. On the other hand, it is difficult to consider stability of cusped periodic waves due to the lack of continuity with respect to initial data and the norm inflation in $H^1_{\rm per}(\mathbb{T})$.

Another interesting problem is to study how the nonlinear instability
of peaked waves complicates the interaction dynamics of multi-peaked solutions.
Asymptotic stability of multi-peaked solitary waves is proven in \cite{Molinet1} (without anti-peakons)
and in \cite{Molinet2} (in the presence of anti-peakons). Collisions of peakons and anti-peakons
lead to jumps of the energy quantities (\ref{conserved-quantities}) \cite{BC1}, see also
relevant results in \cite{Kostenko,HGH,H1,Li,Lund}. Perturbations to multi-peakons are expected to grow in the $W^{1,\infty}$ norm similarly to perturbations to single peakons considered in \cite{NP} and in this work.




\section{Peaked periodic waves as weak solutions}

Let us rewrite the initial-value problem for the CH equation (\ref{CHconv}) in the form:
\begin{equation}
\label{CHevol}
\left\{
\begin{array}{l}
u_t + u u_x + Q[u] = 0, \quad t > 0,\\
u |_{t = 0} = u_0,
\end{array} \right.
\end{equation}
where
\begin{equation}
\label{Q-def}
Q[u](x) := \frac{1}{2} \int_{\mathbb{T}} \varphi'(x-y) q[u](y) dy, \quad
q[u] := u^2 + \frac{1}{2} u_x^2, \quad x \in \mathbb{T}.
\end{equation}
and the time dependence of $Q[u]$ and $q[u]$ is dropped for convenience. The following lemma describes properties of $Q[u]$
depending on the class of functions for $u$.

\begin{lemma}
\label{lem-nonlinear-1}
If $u \in H^1_{\rm per}(\mathbb{T})$, then $Q[u] \in C^0_{\rm per}(\mathbb{T})$.
If in addition, $u \in W^{1,\infty}(\mathbb{T})$, then $Q[u]$ is Lipschitz on $\mathbb{T}$.
\end{lemma}

\begin{proof}
The integration in (\ref{Q-def}) can be split as a sum of two terms:
\begin{equation*}
Q[u](x) = \frac{m}{2} \left[ \int_{-\pi}^{x} \sinh(x-y-\pi) q[u](y) dy - \int_x^{\pi} \sinh(y-x-\pi) q[u](y) dy \right],
\end{equation*}
where $m$ is given by (\ref{minimum}).
Since $q[u]$ is absolutely integrable if $u \in H^1_{\rm per}(\mathbb{T})$, each integral is continuous on $\mathbb{T}$.
If $u \in H^1_{\rm per}(\mathbb{T}) \cap W^{1,\infty}(\mathbb{T})$, then $q[u]$ is also bounded, so that
each integral is Lipschitz on $\mathbb{T}$.
\end{proof}

We say that $u \in C([0,T),H^1_{\rm per}(\mathbb{T})\cap W^{1,\infty}(\mathbb{T}))$
is a weak solution to the initial-value problem (\ref{CHevol}) for some maximal existence time $T > 0$ if
\begin{equation}
\label{CHweak}
\int_{0}^{T} \int_{\mathbb{T}} \left( u \psi_t + \frac{1}{2} u^2 \psi_x - Q[u] \psi \right) dx dt
+ \int_{\mathbb{T}} u_0(x) \psi(0,x) dx = 0
\end{equation}
is satisfied for every test function $\psi \in C^1([0,T] \times \mathbb{T})$ such that $\psi(T,\cdot) = 0$.

We consider the class of peaked periodic wave solutions with a single peak on $\mathbb{T}$
placed at the point $x = \xi(t)$ for every $t \in [0,T)$. Hence
we introduce the following notation:
\begin{equation}
\label{C-1-0}
C^1_{\xi} := \{ u \in H^1_{\rm per}(\mathbb{T})\cap W^{1,\infty}(\mathbb{T}) : \quad u_x \in C(\mathbb{T}\backslash \{\xi\}) \}.
\end{equation}
The following lemma shows that the single peak moves with its local characteristic speed
as in (\ref{characteristic-speed}).

\begin{lemma}
\label{lem-nonlinear-2}
Assume that $u \in C([0,T),H^1_{\rm per}(\mathbb{T})\cap W^{1,\infty}(\mathbb{T}))$ is a weak solution
to the CH equation in the form (\ref{CHweak}) and there exists $\xi(t) \in \mathbb{T}$ for $t \in [0,T)$
such that $u(t,\cdot) \in C^1_{\xi(t)}$ for $t \in [0,T)$. Then, $\xi \in C^1(0,T)$ satisfies
\begin{equation}
\label{peak}
\frac{d \xi}{dt} = u(t,\xi(t)), \quad t \in (0,T).
\end{equation}
\end{lemma}

\begin{proof}
Integrating (\ref{CHweak}) by parts for $x < \xi(t)$ and $x > \xi(t)$ on $\mathbb{T}$
and using the fact that $u(t,\cdot) \in C^0_{\rm per}(\mathbb{T})$ and $u(t,\cdot) \in C^1_{\xi(t)}$ for $t \in [0,T)$, we obtain
the following equations piecewise outside the peak's location:
\begin{equation}
\label{jump-1}
u_t(t,x) + u(t,x) u_x(t,x) + Q[u](t,x) = 0, \quad \pm \left[ x-\xi(t) \right] > 0,  \quad t \in (0,T).
\end{equation}
By Lemma \ref{lem-nonlinear-1}, $Q[u]$ is a continuous function of $x$ on $\mathbb{T}$ for $t \in [0,T)$,
hence it follows from (\ref{jump-1}) that
\begin{equation}
\label{jump-2}
[u_t]^+_- + u(t,\xi(t)) [u_x]^+_-  = 0, \quad t \in (0,T),
\end{equation}
where
$$
[v]^+_- := \lim_{x \to \xi(t)^+} v(t,x) - \lim_{x \to \xi(t)^-} v(t,x)
$$
is the jump of $v$ across the peak location at $x = \xi(t)$. On the other hand,
since $u(t,\cdot) \in C^1_{\xi(t)}$ for $t \in [0,T)$, we differentiate $u(t,\xi(t))$ continuously on both sides
from $x = \xi(t)$ and define
\begin{equation}
\label{jump-3}
\dot{u}^{\pm}(t) := \lim_{x \to \xi(t)^{\pm}} \left[ u_t + \frac{d \xi}{dt} u_x \right], \quad t \in (0,T).
\end{equation}
Since $u(t,\xi(t))$ is continuous for $t \in [0,T)$, we have $\dot{u}^+(t) = \dot{u}^-(t)$ almost everywhere for $t \in (0,T)$.
Therefore, it follows from (\ref{jump-3}) that
\begin{equation}
\label{jump-4}
[u_t]^+_- + \frac{d \xi}{dt} [u_x]^+_- = 0, \quad {\rm a.e.} \; t \in (0,T).
\end{equation}
Since $[u_x]^+_- \neq 0$ if $u \notin C^1_{\rm per}(\mathbb{T})$, then it follows from
(\ref{jump-2}) and (\ref{jump-4}) that $\xi(t)$ satisfies (\ref{peak}) almost everywhere for $t \in (0,T)$.
Since $u \in C([0,T) \times \mathbb{T})$ due to Sobolev embedding of $H^1_{\rm per}(\mathbb{T})$
into $C^0_{\rm per}(\mathbb{T})$, then $u(t,\xi(t)) \in C^0(0,T)$, so that equation (\ref{peak}) is satisfied
everywhere for $t \in (0,T)$ and $\xi \in C^1(0,T)$.
\end{proof}

As a corollary, it follows from Lemma \ref{lem-nonlinear-2} that if $u(t,x) = \varphi(x-ct)$ is the travelling peaked periodic wave,
then $c = \varphi(0) = M$ is uniquely defined with $M$ given by (\ref{maximum})
(compared to the incorrect value $c = 1$ used in \cite{Len1,Len2}).
The following lemma proves that $c = M$ by explicit computation.

\begin{lemma}
\label{lem-traveling}
The traveling peaked periodic wave $u(t,x) = \varphi(x-ct)$ satisfies the stationary equation
\begin{equation}
\label{statCH}
- c \varphi + \frac{1}{2} \varphi^2 + \frac{3}{4} \varphi \ast \varphi^2 = d,
\quad x \in \mathbb{T},
\end{equation}
with $c = M$ and $d = m^2$, where the nonlocal equation is piecewise $C^1$ on both sides from the peak at $x = 0$.
\end{lemma}

\begin{proof}
We shall use the relation $(\varphi')^2 = \varphi^2 - m^2$ for $x \in \mathbb{T} \backslash \{0\}$, 
which follows from (\ref{Green-explicit}). 
Substituting $u(t,x) = \varphi(x-ct)$ into (\ref{CHconv}) and integrating in $x$ yields
the nonlocal equation (\ref{statCH}) where $d$ is an integration constant. In order to verify the validity
of this equation and the explicit values of $c$ and $d$, we shall consider $x \in (0,2\pi)$, for which
one can use the expression $\varphi(x) = m \cosh(\pi - x)$ without the modulus sign. By using $\varphi \ast 1 = 2$
and continuing with explicit evaluation of integrals, we derive
\begin{eqnarray*}
\varphi \ast \varphi^2 & = & \frac{m^2}{4} \left[ \varphi \ast e^{2\pi - 2x} + 2 \varphi \ast 1 + \varphi \ast e^{-2\pi + 2x} \right] \\
& = & \frac{m^2}{3} \left[ 3 + 4 \cosh(\pi) \cosh(\pi - x) - \cosh(2 \pi - 2x) \right],
\end{eqnarray*}
so that
\begin{eqnarray*}
\frac{3}{4} \varphi \ast \varphi^2 + \frac{1}{2} \varphi^2 = M \varphi + m^2,
\end{eqnarray*}
which coincides with (\ref{statCH}) for $c = M$ and $d = m^2$.
\end{proof}

By using Lemmas \ref{lem-nonlinear-2} and \ref{lem-traveling}, we shall now derive
the evolution equations for perturbations near the peaked periodic wave.
We are looking for a weak solution $u \in C([0,T),H^1_{\rm per}(\mathbb{T})\cap W^{1,\infty}(\mathbb{T}))$
to the CH equation in the form (\ref{CHweak}), for which there exists $\xi(t) = c t + a(t) \in \mathbb{T}$ for $t \in [0,T)$
such that $u(t,\cdot) \in C^1_{\xi(t)}$ for $t \in [0,T)$. We present the solution
in the form:
\begin{equation}
\label{decomp}
u(t,x) = \varphi(x-ct-a(t)) + v(t,x-ct-a(t)), \quad t \in [0,T), \quad x \in \mathbb{T},
\end{equation}
where $c = M$, $a(t)$ is the deviation of the peak position from its unperturbed position moving with the speed $c$,
and $v(t,x)$ is the perturbation to the peaked periodic wave $\varphi$.
By Lemma \ref{lem-nonlinear-2}, $a \in C^1(0,T)$ satisfies the equation
\begin{equation}
\label{peak-moves}
\frac{da}{dt} = v(t,0), \quad t \in (0,T).
\end{equation}
Substituting (\ref{decomp}) and (\ref{peak-moves}) into the initial-value problem (\ref{CHevol}) yields
the following problem for the peaked perturbation $v$:
\begin{equation}
\label{CHpert}
\left\{
\begin{array}{l}
v_t = (c - \varphi) v_x + (v|_{x=0} - v) \varphi'  + (v|_{x=0} - v) v_x \\
\qquad \qquad - \varphi' \ast \left(\varphi v + \frac{1}{2} \varphi' v_x \right) - Q[v], \qquad \qquad \qquad t \in (0,T),\\
v |_{t = 0} = v_0,
\end{array} \right.
\end{equation}
where we have used the stationary equation (\ref{statCH}) piecewise on both sides from
the peak and replaced $x - c t - a(t)$ by $x$ due to the translational invariance of the system
(\ref{CHevol}) with the convolution integral (\ref{Q-def}).

\section{Linearized evolution}

Here we study the linearized equation of motion arising in the truncation
of the nonlinear equation in system (\ref{CHpert}) at the linear terms in $v$:
\begin{equation}
\label{linCH}
v_t = (c - \varphi) v_x + (v|_{x=0} - v) \varphi' - \varphi' \ast \left(\varphi v + \frac{1}{2} \varphi' v_x \right).
\end{equation}
We first simplify the linearized equation (\ref{linCH}) by using the following elementary result.

\begin{lemma}
\label{lem-0}
Assume that $v \in H^1_{\rm per}(\mathbb{T})$. Then, it is true for every $x \in \mathbb{T}$ that
\begin{equation*}
[ v(0) - v(x)] \varphi'(x) - (\varphi' \ast \varphi v)(x) - \frac{1}{2} (\varphi' \ast \varphi' v_x)(x)
= \varphi(x) \int_0^x v(y) dy - \frac{1}{2} m^2 \sinh(x) \int_{-\pi}^{\pi} v(y) dy.
\end{equation*}
\end{lemma}

\begin{proof}
Since integrals of absolutely integrable functions are continuous, the map
$$
x \mapsto \varphi' \ast \left( \varphi v + \frac{1}{2} \varphi' v_x \right)
$$
is continuous for every $x \in \mathbb{T}$. Now, $H^1_{\rm per}(\mathbb{T})$ is continuously
embedded into the space of continuous and periodic functions on $\mathbb{T}$, hence $v \in C^0_{\rm per}(\mathbb{T})$.
Integrating by parts yields the following explicit expression for every $x \in \mathbb{T}$:
\begin{eqnarray*}
\left( \varphi' \ast \varphi' v_x\right)(x)
& = & \left( \varphi'' \ast \varphi' v \right)(x) - \left( \varphi' \ast \varphi'' v \right)(x) \\
& = & \left( \varphi \ast \varphi' v \right)(x) - 2 \varphi'(x) v(x) - \left( \varphi' \ast \varphi v \right)(x) + 2\varphi'(x)  v(0),
\end{eqnarray*}
which yields
\begin{equation*}
[ v(0) - v(x)] \varphi'(x) - (\varphi' \ast \varphi v)(x) - \frac{1}{2} (\varphi' \ast \varphi' v_x)(x)
= -\frac{1}{2} (\varphi' \ast \varphi v)(x) - \frac{1}{2} (\varphi \ast \varphi' v)(x).
\end{equation*}
Furthermore, we obtain for $x \in (0,\pi]$:
\begin{eqnarray*}
& \phantom{t} & -\frac{1}{2} \left( \varphi \ast \varphi' v\right)(x) - \frac{1}{2} \left( \varphi' \ast \varphi v\right)(x) \\
& = & -\frac{m^2}{2} \left[ \int_{-\pi}^0 \sinh(x) v(y) dy
- \int_0^x \sinh(2\pi - x) v(y) dy + \int_x^{\pi} \sinh(x) v(y) dy \right] \\
& = & \varphi(x) \int_0^x v(y) dy - \frac{m^2}{2} \sinh(x) \int_{-\pi}^{\pi} v(y) dy,
\end{eqnarray*}
which completes the proof of the equality for $x \in (0,\pi]$. For $x \in [-\pi,0)$, 
the computations are similar:
\begin{eqnarray*}
& \phantom{t} & -\frac{1}{2} \left( \varphi \ast \varphi' v\right)(x) - \frac{1}{2} \left( \varphi' \ast \varphi v\right)(x) \\
& = & -\frac{m^2}{2} \left[ \int_{-\pi}^x \sinh(x) v(y) dy
+ \int_x^0 \sinh(2\pi + x) v(y) dy + \int_0^{\pi} \sinh(x) v(y) dy \right] \\
& = & \varphi(x) \int_0^x v(y) dy - \frac{m^2}{2} \sinh(x) \int_{-\pi}^{\pi} v(y) dy.
\end{eqnarray*}
The zero value at $x = 0$ is recovered by taking the one-sided limits $x \to 0^{\pm}$ in the previous two expressions. 
\end{proof}

By Lemma \ref{lem-0}, we can rewrite the initial-value problem for the linear equation (\ref{linCH})
in the equivalent form:
\begin{equation}
\label{linCH-equiv}
\left\{ \begin{array}{l}
v_t = (c - \varphi) v_x + \varphi w - \pi m^2 \bar{v} \sinh(x), \quad t > 0, \\
v |_{t = 0} = v_0, \end{array} \right.
\end{equation}
where we have introduced
\begin{equation}
\label{w-definition}
w(t,x) := \int_0^x v(t,y) dy, \quad \bar{v}(t) := \frac{1}{2\pi} \int_{-\pi}^{\pi} v(t,y) dy = \frac{w(t,\pi) - w(t,-\pi)}{2\pi}.
\end{equation}
Let us consider the linearized inital-value problem (\ref{linCH-equiv})
in the space $C^1_0$ defined by (\ref{C-1-0}) with $\xi \equiv 0$. The following lemma shows
that the values of $v(t,0)$ and $\bar{v}(t)$ are independent of $t$.

\begin{lemma}
\label{lem-1}
Assume that there exists a solution $v \in C(\mathbb{R}^+,C^1_0)$ to the initial-value problem (\ref{linCH-equiv}).
Then, $v(t,0) = v_0(0)$ and $\bar{v}(t) = \bar{v}_0$ for every $t \in \mathbb{R}^+$.
\end{lemma}

\begin{proof}
If $v \in C(\mathbb{R}^+,C^1_0)$, then $w \in C(\mathbb{R}^+,C^1(\mathbb{T}))$
so that $w(t,0) = 0$ follows from (\ref{w-definition}). Hence, it follows 
from (\ref{linCH-equiv}) that 
$$
\lim_{x \to 0^{\pm}} v_t(t,x) = 0, \quad t \in \mathbb{R}^+,
$$
due to $\varphi(0) = c$ and $v(t,\cdot) \in C^1_0$ for every $t \in \mathbb{R}^+$.
Hence, $v(t,0) = v_0(0)$ for every $t \in \mathbb{R}^+$.

Integrating the evolution equation for the solution $v \in C(\mathbb{R}^+,C^1_0)$ in $x$ on $\mathbb{T}$ 
and using integration by parts piecewise on $[-\pi,0]$ and $[0,\pi]$, we obtain for every $t \in \mathbb{R}^+$:
\begin{eqnarray*}
\frac{d}{dt} \bar{v}(t) & = & \frac{1}{2\pi} \int_{\mathbb{T}} (\varphi(0) - \varphi(x)) v_x(t,x) dx + \frac{1}{2\pi} \int_{\mathbb{T}} \varphi(x) w(t,x) dx\\
 & = & \frac{1}{2\pi} \left[ (\varphi(0) - \varphi(x)) v(t,x) + \varphi'(x) \int_0^{x} v(t,y) dy \right] \biggr|_{x=0^+}^{x = \pi} +
 \biggr|_{x = -\pi}^{x = 0^-}  \\
 & = & 0,
\end{eqnarray*}
due to $\varphi(\pi) = \varphi(-\pi)$, $\varphi'(\pm \pi) = 0$, and $v(t,\cdot) \in C^0_{\rm per}(\mathbb{T})$ for every $t \in \mathbb{R}^+$.
Hence, $\bar{v}(t) = \bar{v}(0)$ for every $t \in \mathbb{R}^+$.
\end{proof}

The evolution problem (\ref{linCH-equiv}) can be solved explicitly by using the method of characteristics
piecewise for $x \in \mathbb{T}$ on both sides of the peak at $x = 0$. The family of characteristic curves
$x = X(t,s)$ satisfies the initial-value problem
\begin{equation}
\label{char}
\left\{ \begin{array}{l}
\frac{dX}{dt} = \varphi(X)-\varphi(0), \\
X |_{t=0} = s, \end{array} \right.
\end{equation}
for every $s \in \mathbb{T}$. 

Along each characteristic curve $x = X(t,s)$ satisfying (\ref{char}), 
function $W(t,s) := w(t,X(t,s))$ satisfies the initial-value problem:
\begin{equation}
\label{w-eq}
\left\{ \begin{array}{l}
\frac{dW}{dt} = \varphi'(X(t,s)) W + \pi m^2 \bar{v} [ 1 - \cosh(X(t,s))], \\
W |_{t=0} = w_0(s), \end{array} \right.
\end{equation}
where $w_0(x) := \int_0^x v_0(y) dy$. The initial-value problem (\ref{w-eq})  
is obtained by linearizing 
the nonlinear equation (\ref{w-appendix}) in Appendix A and posing it at the characteristic  curve $x = X(t,s)$. It follows from (\ref{linCH-equiv}) and (\ref{char}) that
$V(t,s) := v(t,X(t,s))$ satisfies the initial-value problem:
\begin{equation}
\label{v-eq}
\left\{ \begin{array}{l}
\frac{dV}{dt} = \varphi(X(t,s)) W(t,s) - \pi m^2 \bar{v} \sinh(X(t,s)), \\
V |_{t=0} = v_0(s). \end{array} \right.
\end{equation}
The following lemma gives the solution of the initial-value problem
(\ref{linCH-equiv}) in class $C^1(\mathbb{R}^+,C^1_0)$ from
solutions to the initial-value problems (\ref{char}), (\ref{w-eq}), and (\ref{v-eq}).

\begin{lemma}
\label{lem-2}
For every $v_0 \in C^1_0$, there exists the unique solution
$v \in C^1(\mathbb{R}^+,C^1_0)$ to the initial-value problem (\ref{linCH-equiv}).
\end{lemma}

\begin{proof}
By the existence and uniqueness theory for differential equations,
there exists the unique solution $X(\cdot,s) \in C^1(\mathbb{R}^+)$ to (\ref{char}) for every $s \in \mathbb{T}$
due to the Lipschitz continuity of $\varphi$ on $\mathbb{T}$. The peak's location at $X = 0$
is a critical point which remains invariant under the time flow, hence $X(t,0) = 0$.
The ends of $\mathbb{T}$ at $x = \pm \pi$ are not invariant under the time flow.
However, the adjacent peak's location at $X = 2\pi$ is also a critical point invariant under the time flow with $X(t,2\pi) = 2\pi$, hence the interval $[0,2\pi]$ is invariant under the time flow and it makes sense to consider the initial-value problem (\ref{char}) on the interval $[0,2\pi]$.

By the continuous dependence
of solutions on the initial data, the map $(0,2\pi) \ni s \mapsto X(t,s) \in (0,2\pi)$
is $C^1$ for every $t \in \mathbb{R}^+$. Moreover, the transformation is invertible because
\begin{equation}
\label{property-1}
\frac{\partial X}{\partial s} = e^{\int_0^t \varphi'(X(t',s)) dt'} > 0, \quad t \in \mathbb{R}^+, \quad s \in (0,2\pi).
\end{equation}

By substituting the solution $X(t,s)$ to (\ref{char}) 
into (\ref{w-eq}) and solving the linear inhomogeneous equation for $W$,
we obtain the unique solution $W(\cdot,s) \in C^1(\mathbb{R}^+)$ to (\ref{w-eq}) for every fixed $s \in [0,2\pi]$.
Moreover, since $v_0 \in C^0_{\rm per}(\mathbb{T})$, the map $(0,2\pi) \ni s \mapsto W(t,s) \in \mathbb{R}$
is $C^1$ for every $t \in \mathbb{R}^+$. The right-hand side of (\ref{w-eq}) implies
$W(t,0) = 0 = w_0(0)$ and $W(t,2\pi) = 2\pi \bar{v} = w_0(2\pi)$, where the last equality is due to invariance of the mean value of a periodic function in $C^0_{\rm per}(\mathbb{T})$
with respect to any starting point on $\mathbb{T}$. .

Similarly, by substituting the solutions $X(t,s)$ and $W(t,s)$ 
to (\ref{char}) and (\ref{w-eq}) into (\ref{v-eq}) and integrating 
it for $V$, 
we obtain the unique solution $V(\cdot,s) \in C^1(\mathbb{R}^+)$ to (\ref{v-eq}) for every fixed $s \in [0,2\pi]$.
Moreover, since $v_0 \in C^1_0$, the map $(0,2\pi) \ni s \mapsto V(t,s) \in \mathbb{R}$
is $C^1$ for every $t \in \mathbb{R}^+$. The right-hand side of (\ref{v-eq}) implies
$V(t,0) = v_0(0)$ and $V(t,2\pi) = v_0(2\pi) = v_0(0)$. Since $v_0 \in C^1_0$, 
then $V(t,\cdot) \in C^1_0$  for every $t \in \mathbb{R}^+$.

Finally, the change of coordinates $\mathbb{R}^+ \times (0,2\pi) \ni (t,s) \to (t,X) \in \mathbb{R}^+ \times (0,2\pi)$
is a diffeomorphism due to (\ref{property-1}).
Since $\mathbb{T}$ is compact, the solution $v(t,\cdot) = V(t,s = X^{-1}(t,\cdot))$ belongs to $C^1_0$ for every $t \in \mathbb{R}^+$.
\end{proof}

Although the proof of Lemma \ref{lem-2} does not rely on construction of the exact solutions
to the initial-value problems (\ref{char}), (\ref{w-eq}), and (\ref{v-eq}), it is easy to obtain exact solutions 
(see Appendix B). By analyzing the exact solution (\ref{v-sol}), we show that $v(t,\cdot)$ remains bounded in the $L^{\infty}$ norm,
as in the following lemma.

\begin{lemma}
\label{lem-3}
Assume that $v_0 \in C^1_0$ in the initial-value problem (\ref{linCH-equiv}).
Then, there exists $C_0 > 0$ such that
\begin{equation}
\label{bounds-lin-1}
\| v(t,\cdot) \|_{L^{\infty}(\mathbb{T})} \leq C_0, \quad t \in \mathbb{R}^+.
\end{equation}
\end{lemma}

\begin{proof}
By Lemma \ref{lem-2}, the unique solution to the initial-value problem
(\ref{linCH-equiv}) with $v_0 \in C^1_0$ satisfies $v(t,\cdot) \in C^1_0$
for every $t \in \mathbb{R}^+$. By Sobolev's embedding, $v(t,\cdot) \in L^{\infty}(\mathbb{T})$
for every $t \in \mathbb{R}^+$. It remains to obtain the bound (\ref{bounds-lin-1}) uniformly in $t \in \mathbb{R}^+$.

It suffices to consider the exact solution (\ref{v-sol}) from Appendix B for $s \in [0,2\pi]$ and $t \in \mathbb{R}^+$
since
$$
\| v(t,\cdot) \|_{L^{\infty}(\mathbb{T})} = \max_{s \in [0,2\pi]} |V(t,s)|.
$$
Since $v_0, w_0 \in L^{\infty}(\mathbb{T})$, we only need to estimate $Y(t,s)$ in (\ref{v-sol}).
By using elementary transformations, we rewrite $Y(t,s)$ in the form:
\begin{equation*}
Y(t,s) = \frac{(1-e^{-t}) [\sinh(\pi - \frac{s}{2}) \cosh(\pi - \frac{s}{2}) + e^{-t} \sinh(\frac{s}{2}) \cosh(\frac{s}{2})]}{
[\sinh(\pi - \frac{s}{2}) + e^{-t} \sinh(\frac{s}{2})]^2 + 2 (\cosh \pi - 1) e^{-t} \sinh(\frac{s}{2}) \sinh(\pi - \frac{s}{2})}.
\end{equation*}
For $s \in [0,\pi]$, we have
$$
0 \leq Y(t,s) \leq \frac{\sinh(\pi - \frac{s}{2}) \cosh(\pi - \frac{s}{2}) + \sinh(\frac{s}{2}) \cosh(\frac{s}{2})}{\sinh^2(\pi - \frac{s}{2})},
$$
which is bounded on $[0,\pi]$ uniformly in $t \in \mathbb{R}^+$. For $s \in [\pi,2 \pi]$, it follows that
\begin{eqnarray*}
0 \leq Y(t,s) & \leq & \frac{\cosh(\frac{s}{2}) [\sinh(\pi - \frac{s}{2}) + e^{-t} \sinh(\frac{s}{2})] - \sinh(\pi - \frac{s}{2}) 
[\cosh(\frac{s}{2}) - \cosh(\pi - \frac{s}{2})]}{
	[\sinh(\pi - \frac{s}{2}) + e^{-t} \sinh(\frac{s}{2})]^2} \\
& \leq & \frac{\cosh(\frac{s}{2})}{\sinh(\pi - \frac{s}{2}) + e^{-t} \sinh(\frac{s}{2})} =: Z(t,s).
\end{eqnarray*}
Since 
$$
\frac{\partial Z}{\partial s} = \frac{\cosh(\pi) - e^{-t}}{2 [\sinh(\pi - \frac{s}{2}) + e^{-t} \sinh(\frac{s}{2})]^2} > 0,
$$
the map $[\pi,2\pi] \ni s \mapsto Z(t,s) \in \mathbb{R}^+$ is monotonically increasing for $t \in \mathbb{R}^+$ so that 
$$
\max_{s \in [\pi,2\pi]} Z(t,s) = Z(t,2\pi) = M e^{t},
$$
which is exponentially growing in $t \in \mathbb{R}^+$. We also have
\begin{eqnarray}
\nonumber
w_0(s) - \pi m^2 \bar{v} (\cosh s - 1) (1 - e^{-t}) \\
= \int_{2\pi}^s v_0(s') ds'
+ \pi m^2 \bar{v} ( \cosh 2\pi - \cosh s) + \pi m^2 \bar{v} e^{-t} (\cosh s - 1).
\label{tech-eq}
\end{eqnarray}
The third term in (\ref{tech-eq}) multiplied by $Y(t,s)$ is bounded on $[\pi,2\pi]$ by
$$
\pi m^2 |\bar{v}| e^{-t} (\cosh s - 1) Y(t,s) \leq 2 \pi |\bar{v}| e^{-t} Z(t,2\pi) \leq 2 \pi M |\bar{v}|.
$$
The second term in (\ref{tech-eq}) multiplied by $Y(t,s)$ is bounded on $[\pi,2\pi]$ by
\begin{eqnarray*}
\pi m^2 |\bar{v}| ( \cosh 2\pi - \cosh s) Y(t,s) & \leq & 
2 \pi m^2 |\bar{v}| 
\sinh\left(\pi + \frac{s}{2}\right) \sinh\left(\pi - \frac{s}{2}\right) Z(t,s) \\
& \leq & 4 \pi M |\bar{v}| \frac{\sinh(\pi - \frac{s}{2}) \cosh(\frac{s}{2})}{\sinh(\pi - \frac{s}{2}) + e^{-t} \sinh(\frac{s}{2})} \\
& \leq & 4 \pi M |\bar{v}| \cosh(\pi).
\end{eqnarray*}
Finally, since $v_0 \in L^{\infty}(\mathbb{T})$, it follows for every $s \in [\pi,2\pi]$ that
$$
\left| \int_{2\pi}^s v_0(s') ds' \right| \leq 2 \| v_0 \|_{L^{\infty}(\mathbb{T})} \left( \pi - \frac{s}{2} \right)
\leq 2 \| v_0 \|_{L^{\infty}(\mathbb{T})} \sinh\left( \pi - \frac{s}{2} \right).
$$
Hence, the first term in (\ref{tech-eq}) multiplied by $Y(t,s)$ is bounded by the same estimate as the second term is.
All estimates together yield the bound (\ref{bounds-lin-1}).
\end{proof}

The result of Lemma \ref{lem-3} suggests the linear stability of peaked perturbations in the $L^{\infty}$ norm.
The following lemma shows that the perturbations grow in the $W^{1,\infty}$ norm and this growth is generic
at the right side of the peak.

\begin{lemma}
\label{lem-4}
Assume that $v_0 \in C^1_0$ in the initial-value problem (\ref{linCH-equiv}).
Then, it follows that
\begin{equation}
\label{bounds-lin-2}
\lim_{x \to 0^+} v_x(t,x) = e^t v_0'(0^+) + \left(M v_0(0) - \pi m^2 \bar{v}\right) (e^t - 1)
\end{equation}
and
\begin{equation}
\label{bounds-lin-3}
\lim_{x \to 0^-} v_x(t,x) = e^{-t} v_0'(0^-) + \left(M v_0(0) - \pi m^2 \bar{v}\right) (1 - e^{-t}).
\end{equation}
Consequently, if $v_0 \in C^1_{\rm per}(\mathbb{T})$ with either $v_0(0) \neq \frac{m^2}{2M} \bar{v}$ or
$v_0'(0) \neq 0$ or both, then $v(t,\cdot) \notin C^1_{\rm per}(\mathbb{T})$ for almost every $t \in \mathbb{R}^+$.
\end{lemma}

\begin{proof}
Along each characteristic curve $x = X(t,s)$ satisfying (\ref{char}), function $U(t,s) := v_x(t,X(t,s))$ satisfies the initial-value problem:
\begin{equation}
\label{u-eq}
\left\{ \begin{array}{l}
\frac{dU}{dt} = \varphi'(X(t,s)) \left[ W(t,s) - U(t,s) \right] + \varphi(X(t,s)) V(t,s) - \pi m^2 \bar{v} \cosh(X(t,s)), \\
U |_{t=0} = v_0'(s). \end{array} \right.
\end{equation}
The initial-value problem (\ref{u-eq})  
is obtained by linearizing the nonlinear equation (\ref{u-appendix}) in Appendix A 
and posing it at the characteristic curve $x = X(t,s)$.
Coefficients of the differential equation for $U$ are all $C^1(0,2\pi)$. The limits $s \to 0^+$ and $s \to 2\pi^-$ yield
the differential equations
\begin{equation}
\label{dif-eq-1}
\frac{d}{dt} U(t,0^+) = U(t,0^+) + M v_0(0) - \pi m^2 \bar{v}
\end{equation}
and
\begin{equation}
\label{dif-eq-2}
\frac{d}{dt} U(t,2\pi^-) = -U(t,2\pi^-) + M v_0(0) - \pi m^2 \bar{v},
\end{equation}
where we have used  $V(t,0) = V(t,2\pi) = v_0(0)$, $w_0(2\pi) = 2 \pi \bar{v}$, and 
$$
w_0(2\pi) - \pi m^2 \bar{v} \cosh(2\pi) = -\pi m^2 \bar{v}.
$$
Since $U(t,0^-) = U(t,2\pi^-)$ due to periodic continuation of all coefficients of the differential equation (\ref{u-eq}) 
and its solution as $C^1(-2\pi,0)$, we obtain (\ref{bounds-lin-2}) and (\ref{bounds-lin-3}) by solving the linear differential equations
(\ref{dif-eq-1}) and (\ref{dif-eq-2}) with $U(0,0^+) = v_0'(0^+)$ and $U(0,2\pi^-) = U(0,0^-) = v_0'(0^-)$.
\end{proof}

The exponential growth of the $W^{1,\infty}$ norm in Lemma \ref{lem-4}
due to the exact solution (\ref{bounds-lin-2}) at $x = 0^+$ 
is equivalent to the linear instability
of the peaked perturbations in $W^{1,\infty}$. The following lemma
shows that the peaked perturbations are also linearly unstable in $H^1$.

\begin{lemma}
\label{lem-linear}
Assume that $v_0 \in C^1_0$ in the initial-value problem (\ref{linCH-equiv}).
Then, it follows that
\begin{equation}
\label{bound-1}
\| v(t,\cdot) \|_{H^1(\mathbb{T})}^2 = C_+ e^{t} + C_0 + C_- e^{-t},
\end{equation}
for some uniquely defined constants $C_+, C_0, C_-$.
\end{lemma}

\begin{proof}
Since $v(t,\cdot) \in C^1_0$ for every $t \in \mathbb{R}^+$ is the unique solution to the initial-value
problem (\ref{linCH-equiv}) by Lemma \ref{lem-2}, the $H^1$ norm is independent on the choice of the fundamental interval.
Hence we proceed with apriori energy estimates on $[0,2\pi]$ instead of $\mathbb{T}$.
We write the evolution equation for $v$ and $v_x$ explicitly:
\begin{equation}
\label{v-evol}
v_t = (c - \varphi) v_x + \varphi w - \pi m^2 \bar{v} \sinh(x)
\end{equation}
and
\begin{equation}
\label{u-evol}
v_{tx} = (c - \varphi) v_{xx} + \varphi' (w - v_x) + \varphi v - \pi m^2 \bar{v} \cosh(x).
\end{equation}
By multiplying (\ref{v-evol}) by $v$ and (\ref{u-evol}) by $v_x$,
integrating on $[0,2\pi]$, and using integration by parts for $v(t,\cdot) \in C^1_0$, we obtain
\begin{eqnarray*}
\frac{1}{2} \frac{d}{dt} \| v(t,\cdot) \|^2_{L^2(0,2\pi)} & = &
\int_0^{2\pi} \left[ (c - \varphi) v v_x + \varphi v \int_0^x v dy - \pi m^2 \bar{v} \sinh(x) v \right] dx \\
& = & \int_0^{2\pi} \left[ \frac{1}{2} \varphi' v^2 + \varphi v \int_0^x v dy - \pi m^2 \bar{v} \sinh(x) v \right] dx
\end{eqnarray*}
and
\begin{eqnarray*}
\frac{1}{2} \frac{d}{dt} \| v_x(t,\cdot) \|^2_{L^2(0,2\pi)} & = &
\int_0^{2\pi} \left[ (c - \varphi) v_x v_{xx} + \varphi' v_x \int_0^x v dy - \varphi' v_x^2 + \varphi v v_x
- \pi m^2 \bar{v} \cosh(x) v_x \right] dx \\
& = & \int_0^{2\pi} \left[ - \frac{1}{2} \varphi' v_x^2 - \frac{3}{2} \varphi' v^2
- \varphi v \int_0^x v dy  + \pi m^2 \bar{v} \sinh(x) v \right] dx,
\end{eqnarray*}
where we have used $c = \varphi(0) = \varphi(2\pi)$, $\varphi'(2\pi^-) = 1$, and $v(t,0) = v(t,2\pi)$.
Adding these two lines together yields
\begin{equation}
\label{eq-1}
\frac{d}{dt} E(v) = -2 \int_0^{2\pi} \varphi' \left[ v^2 + \frac{1}{2} v_x^2 \right] dx,
\end{equation}
where $E(v) := \| v(t,\cdot) \|^2_{L^2(0,2\pi)} + \| v_x(t,\cdot) \|^2_{L^2(0,2\pi)}$ in agreement with (\ref{conserved-quantities}).

Next, we multiply (\ref{v-evol}) by $\varphi v$ and (\ref{u-evol}) by $\varphi v_x$,
integrate on $[0,2\pi]$, and use integration by parts for $v(t,\cdot) \in C^1_0$.
After straightforward computations, we obtain
\begin{eqnarray*}
\frac{1}{2} \frac{d}{dt} \int_0^{2\pi} \varphi v^2 dx & = &
\int_0^{2\pi} \left[ -\frac{1}{2} (c - 2 \varphi) \varphi' v^2 + \varphi^2 v \int_0^x v dy - \pi m^2 \bar{v} \varphi \sinh(x) v \right] dx
\end{eqnarray*}
and
\begin{eqnarray*}
\frac{1}{2} \frac{d}{dt} \int_0^{2\pi} \varphi v_x^2 dx
& = & \int_0^{2\pi} \left[ - \frac{1}{2} c \varphi' v_x^2 - 2 \varphi \varphi' v^2
- (\varphi^2 + (\varphi')^2) v \int_0^x v dy  \right] dx \\
& \phantom{t} & + \pi m^2 \bar{v} \int_0^{2\pi} \left[ \varphi \sinh(x)
+ \varphi' \cosh(x) \right] v dx,
\end{eqnarray*}
where we can use $(\varphi')^2 = \varphi^2 - m^2$ and 
$\varphi'(x) \cosh(x) - \varphi(x) \sinh(x)  = -1$.
Adding a linear combination of these two lines multiplied by $2$ and $1$ respectively yields
\begin{eqnarray}
\frac{d}{dt} \int_0^{2\pi} \varphi \left[ v^2 + \frac{1}{2} v_x^2 \right] dx = 
- c \int_0^{2\pi} \varphi' \left[ v^2 + \frac{1}{2} v_x^2 \right] dx,
\label{eq-2}
\end{eqnarray}
where we have used that 
$$
\int_0^{2\pi} \left[ v \int_0^x v dy - \frac{1}{2} v \int_0^{2\pi} v dy \right] dx = 0.
$$

Let us define
\begin{equation}
\label{P-S-def}
P(t) := \int_0^{2\pi} \varphi \left[ v^2 + \frac{1}{2} v_x^2 \right] dx, \quad S(t) := \int_0^{2\pi} \varphi' \left[ v^2 + \frac{1}{2} v_x^2 \right] dx.
\end{equation}
It follows from (\ref{eq-1}) and (\ref{eq-2}) that
\begin{equation}
\label{eq-conservation}
c E(v) = 2 P(t) + C_1,
\end{equation}
where $C_1$ is arbitrary constant.

Finally, we multiply (\ref{v-evol}) by $\varphi' v$ and (\ref{u-evol}) by $\varphi' v_x$,
integrate on $[0,2\pi]$, and use integration by parts for $v(t,\cdot) \in C^1_0$.
After straightforward computations, we obtain
\begin{eqnarray*}
\frac{1}{2} \frac{d}{dt} \int_0^{2\pi} \varphi' v^2 dx & = &
\int_0^{2\pi} \left[ -\frac{1}{2} (c \varphi - \varphi^2 - (\varphi')^2) v^2 +
\varphi \varphi' v \int_0^x v dy - \pi m^2 \bar{v} \varphi' \sinh(x) v \right] dx
\end{eqnarray*}
and
\begin{eqnarray*}
\frac{1}{2} \frac{d}{dt} \int_0^{2\pi} \varphi' v_x^2 dx
& = & \int_0^{2\pi} \left[ - \frac{1}{2} (c \varphi - \varphi^2 + (\varphi')^2) v_x^2
- \frac{1}{2} (\varphi^2 + 3 (\varphi')^2) v^2
- 2 \varphi \varphi' v \int_0^x v dy  \right] dx \\
& \phantom{t} & +\pi m^2 \bar{v} \int_0^{2\pi} \left[ \varphi' \sinh(x)
+ \varphi  \cosh(x) \right] v dx + M (v|_{x = 0})^2 - 2 \pi m^2 \bar{v} v|_{x = 0},
\end{eqnarray*}
where we can use $\varphi(x) \cosh(x) -  \varphi'(x) \sinh(x)  = M$. 
Adding a linear combination of these two lines multiplied by $2$ and $1$ respectively yields
\begin{eqnarray}
\label{eq-3}
\frac{d}{dt} S(t) = - c P(t) + \frac{1}{2} m^2 E(v) + C_2,
\end{eqnarray}
where $C_2$ is a constant defined by
$$
C_2 := \pi m^2 M \bar{v} + M (v|_{x = 0})^2 - 2 \pi m^2 \bar{v} v|_{x = 0},
$$
since $v|_{x=0}$ and $\bar{v}$ are independent of $t$.
Equations (\ref{eq-2}) and (\ref{eq-3}) with conservation (\ref{eq-conservation}) yields
the system of differential equations
\begin{equation}
\label{eq-4}
P'(t) = - M S(t), \quad S'(t) = -M^{-1} P(t) + C_3,
\end{equation}
where $C_3 := \frac{m^2 C_1}{2 M} + C_2$ and we have used that $m^2  - M^2 = -1$.
Thus, $S''(t) = S(t)$, so that the most general solution
of the system (\ref{eq-4}) is given by
\begin{equation}
\label{eq-solution}
P(t) = -M S_+ e^t + M S_- e^{-t} + M C_3, \quad S(t) = S_+ e^t + S_- e^{-t},
\end{equation}
where $S_+$ and $S_-$ are arbitrary constants. Substituting (\ref{eq-solution}) into (\ref{eq-conservation})
yields (\ref{bound-1}) for some constants $C_+$, $C_0$, and $C_-$.
\end{proof}

We end this section with two remarks. 
First we show that the conserved quantity (\ref{eq-conservation}) for the linearized evolution 
is obtained from the conserved quantities (\ref{conserved-quantities}) after substituing the solution in the form (\ref{decomp}) 
and expanding in powers of the perturbation. By using (\ref{Green}) and integrating by parts, we obtain
\begin{eqnarray*}
E(u) = E(\varphi) + 2 \int_{\mathbb{T}} (\varphi v + \varphi' v_x) dx + E(v) = E(\varphi) + 4 v|_{x=0} + E(v)
\end{eqnarray*}
and
\begin{eqnarray*}
F(u) & = & F(\varphi) + 2 \int_{\mathbb{T}} \varphi (\varphi v + \varphi' v_x) dx 
+ \int_{\mathbb{T}} v [\varphi^2 + (\varphi')^2] dx + 2 \int_{\mathbb{T}} v (\varphi v + \varphi' v_x) dx \\
&& \qquad \qquad \qquad \qquad  + \int_{\mathbb{T}} \varphi [v^2 + (v_x)^2] dx + F(v) \\
& = & F(\varphi) + 4 M (v|_{x=0}) 
+ 2 \pi m^2 \bar{v}+ 2 P(t) + 2 (v|_{x=0})^2 + F(v),
\end{eqnarray*}
where $P(t)$ is defined by (\ref{P-S-def}). By eliminating $v|_{x=0}$ from the first equality and substituting it to the second equality, we obtain
\begin{equation}
\label{P-E-conservation}
2 P(t) - M E(v) - \frac{1}{4} [E(u) - E(\varphi)] E(v) + \frac{1}{8} [E(v)]^2 + F(v) = C,
\end{equation}
where 
$$
C := F(u) - F(\varphi) - 2 \pi m^2 \bar{v} - M [E(u) - E(\varphi)] - \frac{1}{8} [E(u) - E(\varphi)]^2
$$ 
is constant in the time evolution due to the energy conservation. Neglecting the cubic and quartic terms of $v$ in (\ref{P-E-conservation}) recovers the conserved quantity  (\ref{eq-conservation}) of the linearized evolution since $c = M$.

Finally, we illustrate the exact solutions to the initial-value problems (\ref{char}), (\ref{w-eq}), and (\ref{v-eq}) obtained in Appendix B. 
The solution $v = v(t,x)$ is available in the parametric form 
(\ref{char-sol}) and (\ref{v-sol})  
due to the method of characteristics with parameter $s$ on $[0,2\pi]$. The solution is extended 
periodically to $[-2\pi,0]$. 
Figure \ref{wave-cos} shows the plots of $v(t,x)$ versus $x$ on $[-2\pi,2\pi]$  for different values of $t$ for two initial conditions: $v_0(x) = \sin(x)$ (top) and $v_0(x) = \cos(x)$ (bottom), in both cases, $\bar{v} = 0$. 
These panels give clear illustrations that $v(t,\cdot)$ remains bounded in the $L^{\infty}$ norm (Lemma \ref{lem-3}) and that the slope of the perturbation grows at the right side of the peak (Lemma \ref{lem-4}).

\begin{figure}[htb!]
	\includegraphics[width=0.7\textwidth]{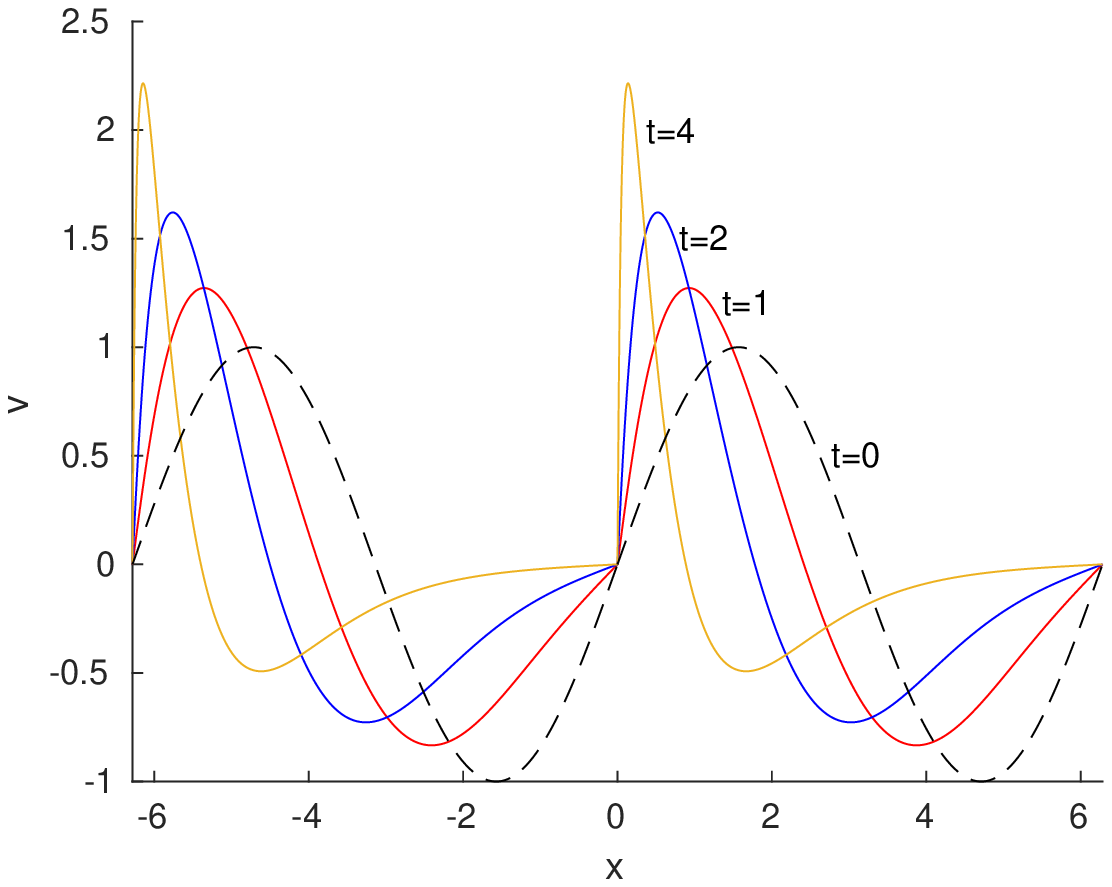}
	\includegraphics[width=0.7\textwidth]{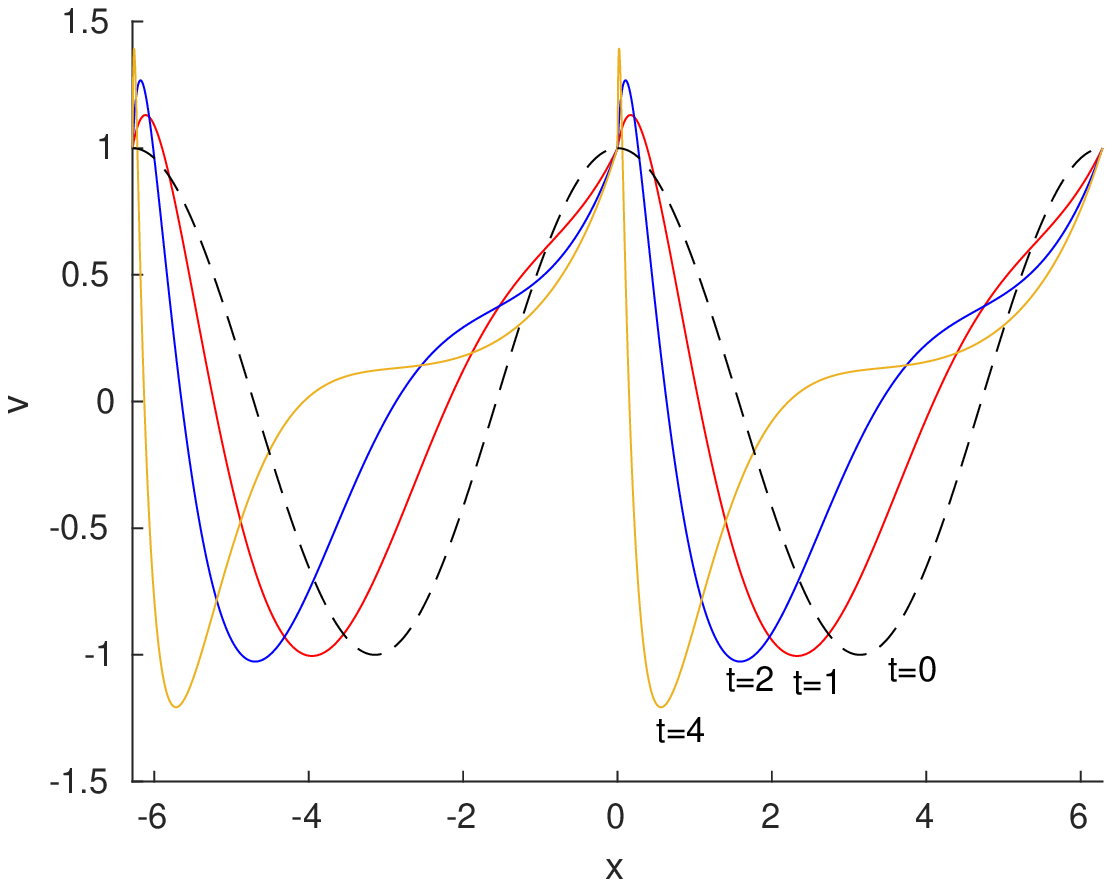}
	\caption{The plots of $v(t,x)$ versus $x$ on $[-2\pi,2\pi]$ for different values of $t$ in the case $v_0(x) = \sin(x)$ (top) and 
		$v_0(x) = \cos(x)$ (bottom). } \label{wave-cos}
\end{figure}

\section{Nonlinear evolution}

Here we analyze the initial-value problem (\ref{CHpert}) with $v_0 \in C^1_0$ and prove Theorem \ref{theorem-nonlinear}. The challenge is that the local well-posedness in $C^1_0$ has not been established in the periodic domain $\mathbb{T}$. Neverthless, we will get the local well-posedness result by using the method of characteristics and the ODE theory.

By simplifying the linear terms with Lemma \ref{lem-0}, 
we rewrite the initial-value problem (\ref{CHpert}) in the equivalent form:
\begin{equation}
\label{v-nonlinear}
\left\{
\begin{array}{l}
v_t = (c - \varphi) v_x + \varphi w - \pi m^2 \bar{v} \sinh(x) 
+ (v|_{x=0} - v) v_x - Q[v],\\
v |_{t = 0} = v_0,
\end{array} \right.
\end{equation}
where $w$ and $\bar{v}$ are defined by the same expressions as in (\ref{w-definition}) and 
\begin{equation}
\label{Q-def-again}
Q[v](x) := \frac{1}{2} \int_{\mathbb{T}} \varphi'(x-y) q[v](y) dy, \quad
q[v] := v^2 + \frac{1}{2} v_x^2, \quad x \in \mathbb{T}.
\end{equation}
As is derived in Appendix A, $w$ satisfies the initial-value problem:
\begin{equation}
\label{w-nonlinear}
\left\{
\begin{array}{l}
w_t = (c - \varphi) w_x + \varphi' w - \pi m^2 \bar{v} [\cosh(x) - 1] 
- \frac{1}{2} (v|_{x=0} - v)^2 - P[v] + P[v] |_{x=0},\\
w |_{t = 0} = w_0,
\end{array} \right.
\end{equation}
where $w_0(x) := \int_0^x v_0(y) dy$ and 
\begin{equation}
\label{P-def}
P[v](x) := \frac{1}{2} \int_{\mathbb{T}} \varphi(x-y) q[v](y) dy, \quad 
q[v] := v^2 + \frac{1}{2} v_x^2, \quad x \in \mathbb{T}.
\end{equation}

Similarly to Lemma \ref{lem-1}, we prove that $\bar{v}$ is independent of $t$. 
At the same time, $v|_{x=0}$ does depend on $t$ in the nonlinear evolution. 

\begin{lemma}
	\label{lem-nonlinear-0}
	Assume that there exists a solution $v \in C([0,T),C^1_0)$ to the initial-value problem (\ref{v-nonlinear})
	Then, $\bar{v}(t) = \bar{v}_0$ for every $t \in [0,T)$.
\end{lemma}

\begin{proof}
	We integrate the evolution equation in system (\ref{v-nonlinear}) for the solution $v \in C([0,T),C^1_0)$ in $x$ on $\mathbb{T}$ and use cancelation of the linear terms in $v$ as in the proof of Lemma \ref{lem-1}. Then, it is true 
	that $\frac{d}{dt} \bar{v}(t) = 0$ if and only if 
\begin{equation}
\label{zero-mean-Q}
	\int_{\mathbb{T}} Q[v](t,x) dx = \frac{1}{2} 
	\int_{\mathbb{T}} \int_{\mathbb{T}} \varphi'(x-y) q[v](t,y) dy dx = 0,
\end{equation}
	where both $\varphi'$ and $q[v]$ are absolutely integrable. Interchanging the integrations by Fubini's theorem and integrating $\varphi'$ piecewise on both sides of the peak yields for every $t \in [0,T)$:
	\begin{eqnarray*}
\int_{\mathbb{T}} \int_{\mathbb{T}} \varphi'(x-y) q[v](t,y) dy dx 
& = & \int_{\mathbb{T}} q[v](t,y) \left(\int_{\mathbb{T}} \varphi'(x-y) dx \right) dy \\
& = & \int_{\mathbb{T}} q[v](t,y) \left( \varphi(\pi - y) - \varphi(-\pi - y) \right) dy \\
& = & 0,
	\end{eqnarray*}
	due to periodicity of $\varphi \in C^0_{\rm per}(\mathbb{T})$. 
	Hence, $\bar{v}(t) = \bar{v_0}$ for every $t \in [0,T)$.
\end{proof}

We can now develop local well-posedness theory of the initial-value 
problem (\ref{v-nonlinear}) by means of the 
method of characteristics. The family of characteristic curves $x = X(t,s)$
satisfies the initial-value problem:
\begin{equation}
\label{char-X}
\left\{
\begin{array}{l}
\frac{dX}{dt} = \varphi(X) - \varphi(0) + v(t,X) - v(t,0), \\
X |_{t = 0} = s,
\end{array} \right.
\end{equation}
for every $s \in \mathbb{T}$. Assuming that $v(t,\cdot) \in C^1_0$ for every $t \in [0,T)$, we can differentiate (\ref{char-X}) piecewise in $s \in \mathbb{T} \backslash \{0\}$ and obtain
\begin{equation}
\label{char-X-der}
\left\{
\begin{array}{l}
\frac{d}{dt} \frac{\partial X}{\partial s} = 
\left[ \varphi'(X) + v_x(t,X) \right]  \frac{\partial X}{\partial s}, \\
 \frac{\partial X}{\partial s} |_{t = 0} = 1.
\end{array} \right.
\end{equation}
with the exact solution
\begin{equation}
\label{char-X-der-solution}
 \frac{\partial X}{\partial s}(t,s) = \exp\left(\int_0^t \left[ \varphi'(X(t',s)) + v_x(t,X(t',s)) \right] dt' \right).
\end{equation}
The peak's locations at $X(t,0) = 0$ and $X(t,2\pi) = 2\pi$ are invariant 
in the time evolution if $v(t,\cdot) \in C^1_0$ for every $t \in [0,T)$.

Along each characteristic curve $x = X(t,s)$ satisfying (\ref{char-X}), functions 
$V(t,s) := v(t,X(t,s))$, $W(t,s) = w(t,X(t,s))$, and $U(t,s) = v_x(t,X(t,s))$ satisfy the following initial-value problems:
\begin{equation}
\label{char-V}
\left\{
\begin{array}{l}
\frac{dV}{dt} = \varphi(X) W - \pi m^2 \bar{v} \sinh(X) - Q[v](X), \\
V |_{t = 0} = v_0(s),
\end{array} \right.
\end{equation}
\begin{equation}
\label{char-W}
\left\{
\begin{array}{l}
\frac{dW}{dt} = \varphi'(X) W 
- \pi m^2 \bar{v} [ \cosh(X) - 1] 
+ \frac{1}{2} [ V^2 - (v|_{X=0})^2 ] - P[v](X) + P[v](0), \\
W |_{t = 0} = w_0(s),
\end{array} \right.
\end{equation}
and
\begin{equation}
\label{char-U}
\left\{
\begin{array}{l}
\frac{dU}{dt} = \varphi'(X) [W - U] + \varphi(X) V - \pi m^2 \bar{v} \cosh(X)
- \frac{1}{2} U^2 + V^2 - P[v](X), \\
U |_{t = 0} = v_0'(s).
\end{array} \right.
\end{equation}
Derivation of (\ref{char-V}), (\ref{char-W}), and (\ref{char-U})  follow 
from (\ref{v-appendix}), (\ref{w-appendix}), and (\ref{u-appendix}) in Appendix A 
after using (\ref{char-X}). The following lemma transfers the local well-posedness theory for differential equations to the initial-value problem (\ref{v-nonlinear}).

\begin{lemma}
\label{lem-nonlinear-3}
For every $v_0 \in C^1_0$, there exists the maximal existence time $T > 0$ (finite or infinite) and the unique solution $v \in C^1([0,T),C^1_0)$ to the initial-value problem (\ref{v-nonlinear}) that depends continuously on $v_0 \in C^1_0$.
\end{lemma}

\begin{proof}
If $v \in C^1_0$, then $Q[v] \in C^0_{\rm per}(\mathbb{T}) \cap {\rm Lip}(\mathbb{T})$ by Lemma \ref{lem-nonlinear-1}. It follows from Lemma 
\ref{lem-nonlinear-0} due to (\ref{zero-mean-Q}) that if $v \in C^1_0$, 
then $P[v] \in C^1_{\rm per}(\mathbb{T})$. Therefore, the nonlocal parts of the initial-value problems (\ref{char-V}), (\ref{char-W}), and (\ref{char-U}) are well-defined 
and can be considered for $X \in [0,2\pi]$ that corresponds to $s \in [0,2\pi]$.

For $s \in [0,2\pi]$, we rewrite the evolution equations in (\ref{char-X}), (\ref{char-V}), (\ref{char-W}), and (\ref{char-U}) as the dynamical system 
\begin{equation*}
\frac{d}{dt} \begin{bmatrix} X \\ V \\ W  \\ U \end{bmatrix} = 
\begin{bmatrix} \varphi(X) - \varphi(0) + V - V|_{s=0} \\ 
\varphi(X) W - \pi m^2 \bar{v} \sinh(X) - Q[v](X) \\ 
\varphi'(X) W - \pi m^2 \bar{v} [ \cosh(X) - 1 ] 
+ \frac{1}{2} [ V^2 - (V|_{s=0})^2 ] - P[v](X) + P[v] |_{s=0} \\
 \varphi'(X) [W - U] + \varphi(X) V - \pi m^2 \bar{v} \cosh(X)
 - \frac{1}{2} U^2 + V^2 - P[v](X) \end{bmatrix}
\end{equation*}	
subject to the initial condition 
\begin{equation*}
\begin{bmatrix} X \\ V \\ W \\ U\end{bmatrix} \biggr|_{t = 0} = 
\begin{bmatrix} s \\ v_0(s) \\ w_0(s) \\ v_0'(s) \end{bmatrix}
\end{equation*}
and the boundary conditions 
\begin{equation*}
\left\{ \begin{array}{ll} 
X(t,0) = 0, & \quad X(t,2\pi) = 2\pi, \\
V(t,0) = V|_{s=0}, & \quad V(t,2\pi) = V|_{s=0}, \\
W(t,0) = 0, & \quad  W(t,2\pi) = 2 \pi \bar{v}, \end{array} \right.
\end{equation*}
where $V|_{s=0}$ satisfies 
$$
\frac{d}{dt} V \biggr|_{s = 0} = -Q[v](0).
$$

Since the vector field of the dynamical system is 
$C^1$ in $(X,V,W,U)$ on $[0,2\pi] \times \mathbb{R}\times \mathbb{R} \times \mathbb{R}$, there exists a unique local solution 
$X(\cdot,s), V(\cdot,s), W(\cdot,s), U(\cdot,s) \in C^1([0,T))$ to the initial-value problem for some maximal existence time $T > 0$. The solution depends continuously 
on the initial data for every $s \in [0,2\pi]$. Moreover, 
since the initial data is $C^1(0,2\pi)$, then 
$X(t,\cdot), V(t,\cdot), W(t,\cdot), U(t,\cdot) \in C^1(0,2\pi)$ for every $t \in [0,T)$. 
The invertibility of the transformation $[0,2\pi] \ni s \mapsto X(t,s) \in [0,2\pi]$ is guaranteed along the solution by 
\begin{equation}
\label{invertibility}
\frac{\partial X}{\partial s}(t,s) = \exp\left(\int_0^t \left[ \varphi'(X(t',s)) + U(t,s) \right] dt' \right) > 0.
\end{equation}
which follows from (\ref{char-X-der-solution}). Since $0$ and $2 \pi$ are 
equilibrium points of (\ref{char-X}), we have $X(t,0) = 0$ and $X(t,2\pi) = 2\pi$. 
Boundary conditions are preserved along the solution due to consistence 
of (\ref{char-W}) and (\ref{char-U}) with the main equation (\ref{char-V}). 
Due to the boundary conditions, the solution $V(t,\cdot) \in C^1(0,2\pi)$ 
is extended to $V(t,\cdot) \in C^1_0$ on $\mathbb{T}$. 
Due to invertibility of the transformation $[0,2\pi] \ni s \mapsto X(t,s) \in [0,2\pi]$, 
we have $v(t,\cdot) \in C^1_0$ for $t \in [0,T)$ and moreover, $v \in C^1([0,T), C^1_0)$.
\end{proof}

The proof of Theorem \ref{theorem-nonlinear} relies on the study of evolution of 
$v_x \in C^1([0,T),C^0(\mathbb{T}\backslash \{0\}))$ at the right side of the peak. 
By Lemma \ref{lem-nonlinear-3}, we are allowed to define the one-sided limits
$U^{\pm}(t) := \lim_{s \to 0^{\pm}} U(t,s)$ for $t \in [0,T)$, 
where $U(t,0^-) = U(t,2\pi^-)$. The functions $U^{\pm} \in C^1(0,T)$ 
satisfy for $t \in (0,T)$:
\begin{equation}
\label{U-0}
\frac{d U^{\pm}}{dt} = \pm U^{\pm} + M V|_{s=0} 
- \pi m^2 \bar{v} - \frac{1}{2} (U^{\pm})^2 + (V|_{s=0})^2 - P[v](0),
\end{equation}
which follow from taking the limits $s \to 0^{\pm}$ in (\ref{char-U}). 
We will now prove Theorem \ref{theorem-nonlinear}.

Due to the decomposition (\ref{decomp}), we can rewrite the initial bound (\ref{initial-bound-theorem}) in the form:
\begin{equation}
\label{initial-bound}
\| v_0 \|_{H^1(\mathbb{T})} + \| v_0' \|_{L^{\infty}(\mathbb{T})} < \delta,
\end{equation}
where $v_0 \in C^1_0 \subset H^1_{\rm per}(\mathbb{T}) \cap W^{1,\infty}(\mathbb{T})$
and $\delta > 0$ is arbitrary small parameter. We first show that there exists $t_0 \in (0,T)$ and $v_0 \in C^1_0$ such that the unique local solution $v \in C^1([0,T),C^1_0)$ to the initial-value problem (\ref{v-nonlinear})  constructed by 
Lemma \ref{lem-nonlinear-3}  satisfies
\begin{equation}
\label{final-bound}
\| v_x(t_0,\cdot) \|_{L^{\infty}(\mathbb{T})} \geq 1.
\end{equation}

It follows from $2 P[v](0) = \int_{\mathbb{T}} \varphi(-y)q[v](y) dy > 0$
that equation (\ref{U-0}) for the upper sign can be estimated by 
\begin{equation}
\label{diff-ineq}
\frac{d U^+}{dt} \leq U^+ + M V|_{s=0} 
- \pi m^2 \bar{v} + (V|_{s=0})^2,
\end{equation}
By Theorem \ref{theorem-CS}, for every small $\varepsilon > 0$,
there exists $\nu(\varepsilon) > 0$ such that 
if $\| v_0 \|_{H^1(\mathbb{T})} < \nu(\varepsilon)$, then
$\| v(t,\cdot) \|_{H^1(\mathbb{T})} < \varepsilon$ for every $t \in [0,T)$.
By Sobolev's embedding, there is a positive constant $C$ such that 
\begin{equation}
\label{bound-on-V0}
\| v(t,\cdot) \|_{L^{\infty}(\mathbb{T})} \leq C \| v(t,\cdot) \|_{H^1} < C \varepsilon.
\end{equation}
By using (\ref{bound-on-V0}), we can simplify (\ref{diff-ineq}) 
for every $\varepsilon \in (0,1]$ to the form:
\begin{equation}
\label{diff-ineq-2}
\frac{d U^+}{dt} \leq U^+ + M C \varepsilon 
+ \pi m^2 C \varepsilon + C^2 \varepsilon^2 \leq U^+ + C_1 \varepsilon,
\end{equation}
for some $\varepsilon$-independent constant $C_1 > 0$. 
Let us assume that the initial data $v_0 \in C^1_0$ satisfies
\begin{equation}
\label{eq-alpha}
\lim_{x \to 0^+} v_0'(x) = -\| v_0' \|_{L^{\infty}(\mathbb{T})} = -2 C_1 \varepsilon,
\end{equation}
The initial bound (\ref{initial-bound}) is consistent with 
(\ref{eq-alpha}) if for every small $\delta > 0$, the small value 
of $\varepsilon$ satisfies
$$
\nu(\varepsilon) + 2 C_1 \varepsilon < \delta,
$$
which just specifies small $\varepsilon$ in terms of small $\delta$. By integrating 
(\ref{diff-ineq-2}) and using (\ref{eq-alpha}), we obtain
\begin{equation*}
U^+(t) \leq e^{t} \left[  U^+(0) + C_1 \varepsilon \right] = - C_1 \varepsilon e^{t}.
\end{equation*}
Hence, for every small $\varepsilon > 0$ there exists a sufficiently large
$t_1 = -\log(C_1 \varepsilon)$ such that $U^+(t_1) \leq -1$.
If $t_1 < T$, the bound (\ref{final-bound}) is true with some $t_0 \in (0,t_1]$ since 
\begin{equation}
\label{lower-bound}
\| v_x(t,\cdot) \|_{L^{\infty}(\mathbb{T})} = 
\| U(t,\cdot) \|_{L^{\infty}(0,2\pi)} \geq |U^+(t)|, \quad t \in [0,T).
\end{equation}
If $t_1 \geq T$, then $T$ is the finite maximal existence time in Lemma \ref{lem-nonlinear-3} for the solution 
$$
X(\cdot,s), V(\cdot,s), W(\cdot,s), U(\cdot,s) \in C^1([0,T)), \quad 
s \in [0,2\pi].
$$ 
By (\ref{invertibility}), we have $X \in C^1([0,T),C^1(0,2\pi))$ 
if and only if $U \in C^1([0,T),C^0(0,2\pi))$. By the bound 
(\ref{bound-on-V0}), we have $W,V \in C^1([0,T),C^0(0,2\pi))$ 
with bounded limits of $\| W(t,\cdot) \|_{L^{\infty}(0,2\pi)}$ 
and $\| V(t,\cdot) \|_{L^{\infty}(0,2\pi)}$ as $t \to T^-$. 
Then, necessarily, if $T < \infty$, we have 
\begin{equation}
\label{blow-up-ode}
\| U(t,\cdot) \|_{L^{\infty}(0,2\pi)} \to \infty \quad 
\mbox{\rm as} \quad t \to T^-,
\end{equation}
so that there exists $t_0 \in (0,T)$ such that  $\| U(t_0,\cdot) \|_{L^{\infty}(0,2\pi)} \geq 1$ and the bound (\ref{final-bound}) is true 
due to (\ref{lower-bound}).

Finally, we show that there exists $v_0 \in C^1_0$ such that the maximal existence time $T$ is finite. 
Due to the bound (\ref{bound-on-V0}), we have for every $\varepsilon \in (0,1)$:
$$
\left| M V|_{s=0} 
- \pi m^2 \bar{v} + (V|_{s=0})^2 - P[v](0) \right|
\leq M C \varepsilon + \pi m^2 C \varepsilon + C^2 \varepsilon^2 
+ \frac{1}{2} M \varepsilon^2 \leq C_2 \varepsilon,
$$
for some $\varepsilon$-independent constant $C_2 > 0$. 
Let $U_0(\varepsilon)$ be the negative root of the quadratic equation 
$$
U - \frac{1}{2} U^2 + C_2 \varepsilon = 0.
$$
It is clear that $U_0(\varepsilon) \in (-C_2 \varepsilon,0)$. Assume that the initial data $v_0 \in C^1_0$ satisfies
\begin{equation}
\label{eq-alpha-new}
\lim_{x \to 0^+} v_0'(x) = -\| v_0' \|_{L^{\infty}(\mathbb{T})} = 2 U_0(\varepsilon).
\end{equation}
The initial bound (\ref{initial-bound}) is consistent with 
(\ref{eq-alpha-new}) if for every small $\delta > 0$, the small value 
of $\varepsilon$ satisfies
$$
\nu(\varepsilon) + 2 |U_0(\varepsilon)| < \delta,
$$
which again specifies small $\varepsilon$ in terms of small $\delta$. If $U^+(0) = 2 U_0(\varepsilon)$, then differential equation (\ref{U-0}) for $U^+$ implies that $\frac{d}{dt} U^+ < 0$, hence $U^+(t) < U^+(0)$ for small positive $t$ and the map $t \mapsto U(t)$ is monotonically decreasing. Let $\overline{U}$ be the supersolution which satisfies:
\begin{equation*}
\label{Ricatti-inequality}
\frac{d \overline{U}}{dt} = \overline{U} - \frac{1}{2} \overline{U}^2 + C_2 \varepsilon.
\end{equation*}
with $\overline{U}(0) = U^+(0) = 2 U_0(\varepsilon)$. It follows by the comparison theory 
for differential equations that $U^+(t) \leq \overline{U}(t)$ for every $t > 0$ for which $U^+(t)$ exists. 
Since there exists a finite $\overline{T} > 0$ such that 
$\overline{U}(t) \to -\infty$ as $t \to \overline{T}^-$, 
then there exists $T \in (0,\overline{T}]$ such that $U^+(t) \to -\infty$ as $t 
\to T^-$, which implies that the blow-up criterion (\ref{blow-up-ode}) is satisfied for finite $T > 0$. 

Both parts of Theorem \ref{theorem-nonlinear} have been proven.

\appendix

\section{Derivation of the evolution equations for $w$ and $v_x$}

We consider the evolution equation for $v$ given by 
\begin{equation}
\label{v-appendix}
v_t = (c - \varphi) v_x + \varphi w - \pi m^2 \bar{v} \sinh(x) 
 + (v|_{x=0} - v) v_x - Q[v],
\end{equation}
where 
$$
Q[v](x) := \frac{1}{2} \int_{\mathbb{T}} \varphi'(x-y) \left[ 
v^2 +  \frac{1}{2}(v_y)^2 \right](y) dy.
$$
Substituting $v = w_x$ into (\ref{v-appendix}) yields
$$
w_{tx} = (c - \varphi) w_{xx}
+ \varphi w - \pi m^2 \bar{v} \sinh(x) 
 +  (v|_{x=0} - w_x) w_{xx} - Q[v].
$$
By integrating this evolution equation in $x$ 
and picking the constant of integration from the boundary conditions $w(t,0) = 0$, we obtain
\begin{equation}
\label{w-appendix}
w_t = (c - \varphi) w_x + \varphi' w - \pi m^2 \bar{v} [\cosh(x) - 1] 
- \frac{1}{2} (v|_{x=0} - v)^2 - P[v] + P[v] |_{x=0},
\end{equation}
where
$$
P[v](x) := \frac{1}{2} \int_{\mathbb{T}} \varphi(x-y) \left[ 
v^2 +  \frac{1}{2}(v_y)^2 \right](y) dy.
$$

Differentiating (\ref{v-appendix}) in $x$ yields
\begin{equation}
\label{u-appendix}
v_{tx} = (c - \varphi) v_{xx} 
+ \varphi' (w - v_x) + \varphi v - \pi m^2 \bar{v} \cosh(x) 
+ (v|_{x=0} - v) v_{xx} + v^2 - \frac{1}{2} (v_x)^2 
- P[v],
\end{equation}
where we have used the relation 
$$
\frac{1}{2} \int_{\mathbb{T}} \varphi''(x-y) \left[ 
v^2 +  \frac{1}{2}(v_y)^2 \right] dy = 
\frac{1}{2} \int_{\mathbb{T}} \varphi(x-y) \left[ 
v^2 + \frac{1}{2} (v_y)^2 \right] dy - v^2 - \frac{1}{2} v_x^2
$$

\section{Explicit solutions to the linearized equations}

Separation of variables in (\ref{char}) gives
$$
m t = \int_s^X \frac{dx}{\cosh(\pi - x) - \cosh(\pi)} = m \log \frac{(e^{2\pi} - e^X) (1 - e^s)}{(1 - e^X) (e^{2\pi} - e^s)},
$$
from which we derive
\begin{equation}
\label{char-sol}
X(t,s) = \log \frac{(e^{2\pi} - e^s) + e^{2\pi - t} (e^s - 1)}{(e^{2\pi} - e^s) + e^{-t} (e^s - 1)}, \quad t \in \mathbb{R}^+, \quad s \in [0,2\pi],
\end{equation}
such that
$$
\lim_{s \to 0^+} X(t,s) = 0 \quad \mbox{\rm and} \quad \lim_{s \to 2 \pi^-} X(t,s) = 2\pi.
$$
After lengthy manipulations,
we can find that
\begin{equation*}
\frac{\partial X}{\partial s}(t,s) = \frac{(e^{2\pi}-1)^2 e^{s-t}}{
[(e^{2\pi}-e^s) + e^{-t} (e^s - 1)][(e^{2\pi}-e^s) + e^{2\pi - t} (e^s - 1)]}, \quad t \in \mathbb{R}^+, \quad s \in [0,2\pi],
\end{equation*}
such that
$$
\lim_{s \to 0^+} \frac{\partial X}{\partial s}(t,s) = e^{-t} \quad \mbox{\rm and} \quad
\lim_{s \to 2 \pi^-} \frac{\partial X}{\partial s}(t,s) = e^t.
$$

Integrating (\ref{w-eq}) with an integrating factor yields
$$
W(t,s) = \frac{\partial X}{\partial s} \left[ w_0(s) + \pi m^2 \bar{v}
\int_0^t \frac{1 - \cosh X(t',s)}{\frac{\partial X}{\partial s}(t',s)} dt' \right].
$$
By using (\ref{char-sol}), we derive the following simple expression:
\begin{equation}
\label{w-sol}
W(t,s) = \frac{\partial X}{\partial s} \left[ w_0(s) - \pi m^2 \bar{v}
 (\cosh s - 1) (1 - e^{-t}) \right], \quad t \in \mathbb{R}^+, \quad s \in [0,2\pi],
\end{equation}
such that
$$
\lim_{s \to 0^+} W(t,s) = 0 \quad \mbox{\rm and} \quad
\lim_{s \to 2 \pi^-} W(t,s) = w_0(2 \pi) = 2 \pi \bar{v}.
$$

Instead of integrating (\ref{v-eq}), we can use the chain rule
$$
\frac{\partial W}{\partial s}(t,s) = V(t,s) \frac{\partial X}{\partial s}(t,s),
$$
from which we derive
\begin{equation}
\label{v-sol}
V(t,s) = v_0(s) - \pi m^2 \bar{v} \sinh s (1 - e^{-t}) +
\left[ w_0(s) - \pi m^2 \bar{v} (\cosh s - 1) (1 - e^{-t})\right] Y(t,s)
\end{equation}
with
\begin{eqnarray*}
Y(t,s) & := & \frac{\partial}{\partial s} \log \frac{\partial X}{\partial s}(t,s) \\
& = & \frac{(1-e^{-t}) [\sinh(2\pi - s) + e^{-t} \sinh(s)]}{\cosh(2 \pi - s)
-1 + [\cosh(2\pi) + 1 - \cosh(2\pi - s) - \cosh(s)] e^{-t} + (\cosh s - 1) e^{-2t}},
\end{eqnarray*}
such that
$$
\lim_{s \to 0^+} V(t,s) = v_0(0) \quad \mbox{\rm and} \quad
\lim_{s \to 2 \pi^-} V(t,s) = v_0(2 \pi) = v_0(0).
$$

\section{Smooth, peaked, and cusped periodic waves}

Traveling wave solutions to the CH equation (\ref{CH}) are of the form 
$u(t,x) = \phi(x-ct)$, where $c$ is the wave speed and $\phi(x) : \mathbb{T}\mapsto \mathbb{R}$ satisfies the third-order differential equation 
\begin{equation}
\label{third-order}
(c-\phi) (\phi''' - \phi') + 2 \phi \phi' - 2 \phi' \phi'' = 0.
\end{equation}
Multiplying (\ref{third-order}) by $(c - \phi)$ allows us to integrate it once:
\begin{equation}
\label{second-order}
\frac{d}{dx} \left[ (c - \phi)^2 (\phi'' - \phi) \right] = 0 \quad \Rightarrow \quad (c - \phi)^2 (\phi'' - \phi) = a,
\end{equation}
where $a \in \mathbb{R}$ is the integration constant. The level $\phi = c$ is singular and in what follows, we consider $c > 0$. For every $\phi \neq c$, 
the second-order differential equation (\ref{second-order}) can be integrated once:
\begin{equation}
\label{first-order}
\phi'' - \phi - \frac{a}{(c - \phi)^2} = 0 \quad \Rightarrow \quad 
(\phi')^2 - \phi^2 - \frac{2a}{c-\phi} = b,
\end{equation}
where $b \in \mathbb{R}$ is another integration constant. 

Let $W(\phi) := -\phi^2 - 2a/(c-\phi)$. If $a = 0$, there is only one critical point of $W$ at $\phi = 0$ so that $(0,0)$ is a saddle point of the dynamical system 
on the phase plane $(\phi,\phi')$. Peaked periodic waves are constructed from orbits of the linear equation $\phi'' - \phi = 0$ intersecting the singularity line $\phi = c$ over the length of $2 \pi$. If 
$$
M_{\phi} := \max_{x \in \mathbb{T}} \phi(x), \quad 
m_{\phi} := \min_{x \in \mathbb{T}} \phi(x),
$$
then $b = -m_{\phi}^2$ and $c = M_{\phi}$. The exact $2\pi$-periodic solution to $\phi'' - \phi = 0$ is given by 
\begin{equation}
\label{wave-appendix}
\phi(x) = m_{\phi} \cosh(\pi - |x|), \quad M_{\phi} = m_{\phi} \cosh(\pi).
\end{equation}
The uniquely selected peaked wave (\ref{unique-wave}) corresponds to $\phi = \varphi$ with $M_{\phi} = M$ and $m_{\phi} = m$ given by (\ref{maximum}) and (\ref{minimum}). The one-parameter 
family (\ref{wave-appendix}) corresponds to the one-parameter family of peaked $2\pi$-periodic waves (\ref{peaked-traveling-wave}), this family is given by 
$\phi = \gamma \varphi$ with $\gamma := \frac{m_{\phi}}{m}$ being a free parameter and $c = M_{\phi} = \gamma M$ in agreeement between (\ref{peaked-traveling-wave}) and (\ref{wave-appendix}).

If $a > 0$, there is only one critical point of $W$ at some $\phi_0 < 0$ 
so that $(\phi_0,0)$ is a saddle point on the phase plane $(\phi,\phi')$. 
The periodic waves are constructed from orbits squeezed between the stable and unstable curves from $(\phi_0,0)$ and the singularity line $\phi = c$. These orbits give the cusped periodic waves with the following local behavior near the singularity:
$$
\phi(x) \sim c - \alpha x^{2/3} \quad \mbox{\rm as} \quad x \to 0,
$$ 
where $\alpha > 0$ is related to parameter $a > 0$. 

If $a < 0$, there are three critical points of $W$ given by $\{ \phi_1,\phi_2,\phi_3\}$ ordered as $\phi_1 < \phi_2 < c < \phi_3$ so that 
$(\phi_1,0)$ and $(\phi_3,0)$ are saddle points, whereas 
$(\phi_2,0)$ is a center point.
Smooth periodic waves are located inside the homoclinic loop connected to $(\phi_1,0)$ which surrounds $(\phi_2,0)$. 
Cusped waves also exist for $\phi > c$ but not for $\phi < c$.

\end{document}